\documentclass[11pt]{article}

\usepackage{amsmath}
\usepackage{amsthm}
\usepackage{amsfonts}
\usepackage{amssymb}
\usepackage{latexsym}
\usepackage{mathrsfs}

\usepackage[noautoscale]{youngtab}

\usepackage{fullpage}

\usepackage{color, colortbl}
\usepackage{array}
\usepackage{multirow}

\usepackage{graphicx}
\usepackage{subcaption}
\usepackage{tikz}

\usepackage{etoolbox}

\usepackage{enumitem}
\setlist[enumerate, 1]{label=(\roman*)}

\numberwithin{equation}{section}

\newtheorem{thm}{Theorem}[section]
\newtheorem{lem}[thm]{Lemma}
\newtheorem{prop}[thm]{Proposition}
\newtheorem{cor}[thm]{Corollary}
\newtheorem{conj}[thm]{Conjecture}
\newtheorem{rem}[thm]{Remark}

\newtheorem{exa}[thm]{Example}
\newtheorem{fact}[thm]{Fact}
\newtheorem{prob}[thm]{Problem}

\def\bN{\mathbb N}

\def\bP{\mathbb P}
\def\cA{\mathcal{A}}
\def\cE{\mathcal{E}}
\def\cM{\mathcal{M}}

\def\fS{\mathfrak{S}}

\usepackage[colorlinks=true,
linkcolor=webgreen,
filecolor=webbrown,
citecolor=webgreen]{hyperref}

\definecolor{webgreen}{rgb}{0,.5,0}
\definecolor{webbrown}{rgb}{.6,0,0}

\usepackage[symbol]{footmisc}

\begin{document}

\title{Stable patterns on permutations of multisets$^\dag$}

\author{Shaoshi Chen$^\ddag$, Hanqian Fang$^{*,\S}$\ \ and Sergey Kitaev$^{\P}$}

\footnotetext[1]{Corresponding author.}
\footnotetext[2]{Shaoshi Chen and Hanqian Fang were partially supported by the Strategic Priority Research Program of the Chinese Academy of Sciences (No. XDB0510201) and the NSFC grant (No. 12271511).}
\footnotetext[3]{KLMM, Academy of Mathematics and Systems Science, Chinese Academy of Sciences, Beijing 100190, P. R. China. Email: schen@amss.ac.cn.}
\footnotetext[4]{KLMM, Academy of Mathematics and Systems Science, Chinese Academy of Sciences, Beijing 100190, P. R. China.  Email: fanghanqian22@mails.ucas.ac.cn.}
\footnotetext[5]{Department of Mathematics and Statistics, University of Strathclyde, 26 Richmond Street, Glasgow G1 1XH, United Kingdom. Email: sergey.kitaev@strath.ac.uk.}

\date{\today}

\maketitle

\noindent\textbf{Abstract.} In this paper, we study patterns on permutations of multisets whose multivariate distribution generating functions are symmetric. We interpret this phenomenon through the lens of group actions and define such a pattern as stable. Although various stability results are already implicit in existing enumerative work, we explicitly summarize them here and provide bijective proofs. These bijections offer new combinatorial insight into the symmetry of the generating functions. We also establish instability results. In particular, we provide a complete characterization of stable classical patterns, showing that the only such patterns are those of length one or two. For consecutive patterns, we reprove the stability of all monotone patterns and also identify a large class of unstable patterns. We conjecture that monotone patterns are the only stable consecutive patterns. As an application, we use stability to derive recurrence relations for the ascent distribution over permutations of restricted multisets, yielding a generalization of Eulerian numbers.\\

\noindent {\bf AMS Classification 2010:} 05A05; 05A15; 05A19 


\tableofcontents


\newpage

\section{Introduction}
The generating function for a combinatorial property over all words from the set of positive integers is the multivariate series
\[\sum_{n\geq 0}\sum_{k_1,k_2,\ldots,k_{n-1}\geq0,k_n>0}c_{k_1,k_2,\ldots,k_n} x_1^{k_1}x_2^{k_2}\cdots x_n^{k_n},\]
where the coefficient $ c_{k_1, k_2,\ldots, k_n} $ counts the words satisfying the property in which each letter $ i $ appears exactly $ k_i $ times. MacMahon~\cite{MacMahon1915} proposed refining generating functions by substituting $ x_1^{k_1} x_2^{k_2}\cdots x_n^{k_n} $ with the symmetric function it generates, since the symmetric function not only reduces redundant parameters in enumeration, but also reveals deeper algebraic structure. For instance, Gessel~\cite{Gessel1990} established several criteria for determining whether a symmetric generating function is D-finite; moreover, the symmetry of the generating functions enabled Yang and Zeilberger~\cite{Yang2020} to develop an $O(n^{s+1})$ algorithm for enumerating words with $n$ letters each repeated $s$ times while avoiding increasing consecutive patterns of any fixed length.
These considerations motivate our search for symmetric generating functions. A further motivation arises from the study of $3$-dimensional permutations in~\cite[Section 3]{Chen2025}, where the authors reduce these permutations to the words in which letters repeat at most twice. Enumerative results for statistics over these permutations can then be derived from corresponding results over words, provided these word enumerations depend only on the word length and the number of repeated letters. In other words, the generating functions of these statistics are symmetric, at least when restricted to the homogeneous components where the degree of each variable is at most two.

MacMahon derived symmetric generating functions for the distribution of several word statistics via symmetric function theory, including the greater index (now commonly known as the ``major index'')~\cite[Item 105]{MacMahon1915}, the descent number~\cite[Item 155]{MacMahon1915}, and the bivariate statistic of descent number and greater index~\cite{MacMahon1916}. These three types of statistics were subsequently shown to be equidistributed with numerous other statistics, and are now known respectively as \emph{Mahonian},
\emph{Eulerian}, and \emph{Eulerian--Mahonian} statistics, see~\cite{Carnevale2017} for a review. 
Symmetric generating functions also arise frequently in the enumeration of patterns. Yang and Zeilberger~\cite{Yang2020} derived an explicit expression for the generating functions for the distribution of increasing consecutive patterns over permutations of multisets in terms of elementary symmetric functions. Atkinson, Linton, and Walker~\cite{Atkinson1995} gave an explicit symmetric generating function for permutations of multisets avoiding $132$, while Albert et al.~\cite{Albert2001} obtained analogous results for $123$-avoiding permutations and extended them to all classical monotone patterns of arbitrary length. Taken together, these two results show that the generating function for the avoidance of any classical pattern of length $3$ is symmetric. The connection between $123$-avoiding and $132$-avoiding permutations of multisets was later established by Zeilberger~\cite{Zeilberger2005} (see~\cite[Section 3.2]{Shar2016} for a survey). Savage and Wilf~\cite{Savage2006} provided another novel proof of the symmetry for $123$-avoiding permutations of multisets. Their key insight was that this symmetry follows from the invariance of the number of such permutations under the action of the symmetric group. Since the symmetric group is generated by adjacent transpositions, it suffices to construct, for each such transposition, a bijection that maps the set of $123$-avoiding permutations onto the set of permutations that both avoid $123$ and are obtained by applying that transposition.

Inspired by Savage and Wilf, we reinterpret such symmetry through the lens of group action. This classical perspective builds a natural bridge between algebra and combinatorics. It introduces the fundamental concepts in a conceptually transparent manner and also provides a powerful framework for enumeration~\cite{Stanley1982,Garsia1984}. In particular, invariants under the permutation of entries play a key role in various contexts, for example in the enumeration of weakly increasing parking functions~\cite{Harris2025}. We define a property to be \emph{stable} if its corresponding generating function is symmetric, and a pattern to be \emph{stable} if the generating function for its distribution is symmetric. From an algebraic viewpoint, stability means independence under permutations of multiplicities in multisets (for example, the property or statistic behaves identically on permutations of the multisets $\{1,1,2,2,2,3\}$ and $\{1,1,1,2,3,3\}$). 

In this paper, we focus on the characterization of stable patterns. While stability has appeared implicitly in several earlier enumerative studies, we make these instances explicit here and supply bijective proofs. These bijections shed new light on the combinatorial origins of the generating function symmetry. To make our theory more complete, we also establish the instability of a wide class of patterns. In summary, for classical patterns, we prove that the only stable ones are those of length at most two (Theorem~\ref{THM:classic-pattern-charact}); for consecutive patterns, we reprove that all monotone patterns are stable via bijections (Theorem~\ref{THM:mono-cons-stable}); furthermore, Theorem~\ref{PROP:unstable-by-extended} constructs a large class of unstable consecutive patterns, thereby making significant progress toward a complete classification.

As an application of pattern stability, this paper derives some recurrence relations for the distribution of stable pattern $\underline{12}$ over permutations subject to certain constraints. A classical result in this field, dating back to~\cite{MacMahon1915,Carlitz1954}, is that usual permutations with $s$ occurrences of $\underline{12}$ are counted by the {\em Eulerian numbers}. For permutations over a general multiset, MacMahon~\cite[Item 155]{MacMahon1915} provided two algebraic expressions whose coefficients encode the ascent distribution. Subsequent research has focused on the generalized Eulerian polynomial $\cA_{M}(x)$, which records the ascent distribution across all permutations of a fixed multiset $M$. Simion~\cite{Simon1984} proved the real-rootedness of $\mathcal{A}_M(x)$, which implies the log-concavity and unimodality of its coefficient sequence. Refinements of $\cA_{M}(x)$ have been studied in contexts such as mixed volumes~\cite{Stanley1981,Ehrenborg1998} and barycentric subdivisions~\cite{Brenti2008,Nevo2011}, and their analytical properties have been further investigated in~\cite{Savage2015,Ma2023}. 

The paper is organized as follows. Section~\ref{SEC:pre} introduces the concept of stability from the viewpoint of group actions and summarizes the algebraic background for our bijective proofs. In Section~\ref{SEC:classical}, we completely determine all stable classical patterns. Section~\ref{SEC:consecutive} studies consecutive patterns, exhibiting both stable and large unstable families. Section~\ref{SEC:application} uses the stability of the pattern $ \underline{12} $ to derive recurrence relations for its distribution over permutations of a restricted multiset. We conclude in Section~\ref{SEC:conclusion} with some open problems.

\section{Preliminaries}\label{SEC:pre}

This section presents the basic definitions and notations for multisets and patterns. We then introduce the concepts of a stable property and a stable pattern, along with some immediate observations. Finally, we establish the necessary algebraic background to demonstrate how our bijective proofs operate in the subsequent sections.

Throughout this paper, let $\bN$ denote the set of non-negative integers and $\bP$ the set of positive integers. For a vector $(k_1,k_2,\ldots,k_n)\in\bN^n$, let $M(k_1,k_2,\ldots,k_n)$ denote the \textit{multiset} $$\{\underbrace{1,\ldots,1}_{k_1},\underbrace{2,\ldots,2}_{k_2},\ldots,\underbrace{n,\ldots,n}_{k_n}\}$$ on $\bP$, where each $i$ appears with \textit{multiplicity} $k_i$. For a multiset $M=M(k_1,k_2,\ldots,k_n)$, the integer $k_1+k_2+\cdots+k_n$ is called the \textit{size} of $M$, denoted by $|M|$. A \textit{permutation} (also known in the literature as a \textit{multipermutation}) of $M$ is a sequence of \textit{length} $|M|$ where each \textit{letter} $i$ occurs exactly $k_i$ times for all $1\leq i\leq n$. We use $M^*$ to denote the set of all permutations of $M$. For a permutation $\pi = \pi_1 \pi_2 \cdots \pi_m$, a subsequence $\pi_i \pi_{i+1} \cdots \pi_{\ell}$ with $1 \le i \le \ell \le m$ is called a \emph{factor} of $\pi$.
The \emph{reduction} of $\pi$ is obtained by
replacing occurrences of the smallest letter by $1$, those of the next smallest by $2$, and so on. 
For a fixed positive integer $n$, let $\mathcal{M}_n$ be the family of all multisets on $\{1,\ldots,n\}$, i.e.,
$\mathcal{M}_n:=\{M(k_1,\ldots,k_n)\ |\  k_1,\ldots,k_n\in\bN\}$.
Let $\Pi_n$ denote the set of all permutations of all multisets in $\mathcal{M}_n$, i.e., $\Pi_n:=\bigcup_{M\in\mathcal{M}_n} M^*$.
The symmetric group $\mathfrak{S}_n$ acts on $\mathcal{M}_n$ by
\[\sigma\cdot M(k_1,\ldots,k_n)=M(k_{\sigma^{-1}(1)},\ldots,k_{\sigma^{-1}(n)})\]
for $\sigma\in\mathfrak{S}_n$ and $M(k_1,\ldots,k_n)\in \cM_n$. This action induces an action on $\Pi_n$ given by
$\sigma\cdot(\pi_1\pi_2\cdots\pi_m)=\sigma(\pi_1)\sigma(\pi_2)\cdots\sigma(\pi_m)$
for $\sigma\in\mathfrak{S}_n$ and $\pi_1\pi_2\cdots\pi_m\in\Pi_n$.

Let $P$ be a property defined on the permutations of $M=M(k_1,k_2,\ldots,k_n)$. For example, $P$ can be ``to be weakly increasing", ``to be unimodal", ``to have $s$ descents" for some fixed $s\in\bN$, and so on. Let $M^*(P)$ denote the subset of $M^*$ consisting of all permutations satisfying $P$. If the size of the set $M^*(P)$ is invariant under $\fS_n$, that is, if $|M^*(P)|=|(\sigma\cdot M)^*(P)|$ for all $\sigma\in\fS_n$, then we say that the property $P$ is \textit{stable on $M$}. Otherwise, we say it is \textit{unstable on $M$}. Furthermore, if $P$ is stable on any multiset $M(k_1,k_2,\ldots,k_n)$, we say that $P$ is \textit{stable}. From the perspective of the generating functions, $P$ is stable if and only if the corresponding generating function
\[G_P(x_1,x_2,\ldots):=\sum_{n\geq 0}\sum_{k_1,k_2,\ldots,k_{n-1}\geq0,k_n>0}\left|\left(M(k_1,k_2,\ldots,k_n)\right)^*(P)\right|\cdot x_1^{k_1}x_2^{k_2}\cdots x_n^{k_n}\]
is a symmetric function in the infinitely many variables $x_1,x_2,\ldots$. For instance, it is obvious that the property $P=$ ``to be weakly increasing" is stable since there is exactly one permutation in $M^*(P)$ for any multiset $M$. If $P$ is not stable on some multiset, we say that $P$ is \textit{unstable}. An example of an unstable property is ``to be up-down". This follows from the observation that the permutation $121$ in $\{1,1,2\}^*$ is up-down, whereas no up-down permutation exists in $\{1,2,2\}^*$. 

In particular, the joint distribution of several statistics is a prominent topic in the study of permutations. For a tuple of statistics \texttt{st}=(\texttt{st}$_1$,\texttt{st}$_2$,$\ldots$,\texttt{st}$_\mu$), we say that \texttt{st} is \textit{stable} if the property ``to have $s_i$ occurrences of \texttt{st}$_i$ for each $i=1,2,\ldots,\mu$" is stable for any $s_1,s_2,\ldots,s_\mu\geq0$.
Pattern occurrences form one of the most important classes of permutation statistics and serve as a primary source of examples in this framework. 
A \emph{pattern} is a reduced word $p=p_1\cdots p_\ell$. 
We say that a permutation $\pi=\pi_1\pi_2\cdots \pi_m\in M^*$ \textit{contains} an occurrence of the pattern $p$ if there is a subsequence $\pi'=\pi_{i_1}\pi_{i_2}\cdots \pi_{i_\ell}$ whose reduction is $p$. If $\pi$ contains no occurrences of $p$, we say that $\pi$ \textit{avoids} $p$. In an occurrence of a \textit{consecutive pattern}, we further require that the letters of the subsequence $\pi'$ stay next to each other (i.e., $i_2=i_1+1$, $i_3=i_2+1$, etc in $\pi'$). Following the notation of~\cite{Kitaev2011}, we distinguish the consecutive patterns by underlining them, like $\underline{11}$, $\underline{12}$, $\underline{21}$, $\underline{132}$, etc. In the literature, occurrences of the patterns $\underline{11}$, $\underline{12}$, and $\underline{21}$ are referred to as \textit{plateaux} (or \textit{levels}), \textit{ascents}, and \textit{descents}, respectively. The patterns $12\cdots\ell$, $\underline{12\cdots\ell}$, $\ell(\ell-1)\cdots1$, and $\underline{\ell(\ell-1)\cdots1}$ are called \textit{monotone} patterns. A \emph{vincular pattern} generalizes the notions of the patterns introduced above~\cite{Kitaev2011}. In an occurrence of a vincular pattern, some consecutive letters may be required to stay next to each other, while other pairs of consecutive letters may not. For example, in an occurrence of the vincular pattern $\underline{21}3$, the letters corresponding to 2 and 1 must be adjacent, while the letter corresponding to 3 is to the right of that corresponding to 1 (without any further restrictions). 

For a pattern $p$ and a permutation $\pi$, let $p(\pi)$ denote the number of occurrences of $p$ in $\pi$. For $s\in\bN$, let $M^*(p;s)$ denote the subset of $M^*$ consisting of all the permutations with $s$ occurrences of the pattern $p$. For instance, the permutation $\pi=211342\in M(2,2,1,1)$ avoids the pattern $\underline{132}$ while it contains three occurrences of the pattern $123$: the subsequences $234$ (occurring once) and $134$ (occurring twice). Thus $\underline{132}(\pi)=0$, $123(\pi)=3$, and $\pi\in M^*(\underline{132};0)\cap M^*(123;3)$. Unless otherwise specified, all patterns considered in this paper are assumed to have distinct letters.

Many properties on permutations of multisets can be restated in terms of patterns. For example, the property ``to be weakly increasing" is equivalent to ``to avoid the pattern $21$"; the property ``to be alternating" means ``to avoid the patterns $\underline{123}$, $\underline{321}$, and $\underline{11}$", and so on. For a given multiset $M$, we say that a pattern $p$ is \textit{stable on $M$} if the property ``to have $s$ occurrences of $p$" for each $s\in\bN$ is stable on $M$, i.e., $|M^*(p;s)|=|(\sigma\cdot M)^*(p;s)|$ for all $\sigma\in\fS_n$. Also, a pattern $p$ is \textit{unstable on $M$} if it is not stable on $M$. Finally, if a pattern $p$ is stable on any multiset $M(k_1,k_2,\ldots,k_n)$, we say that it is \textit{stable}. In other words, the pattern $p$ is stable if and only if its generating function
\[G_p(x_0,x_1,x_2,\ldots):=\sum_{n\geq 0}\sum_{k_1,k_2,\ldots,k_{n-1}\geq0,k_n>0}\sum_{s\geq 0}\left|\left(M(k_1,k_2,\ldots,k_n)\right)^*(p;s)\right| \cdot x_0^sx_1^{k_1}x_2^{k_2}\cdots x_n^{k_n}\]
is symmetric with respect to the infinitely many variables $x_1,x_2,\ldots$. If there is a multiset $M$ on which $p$ is unstable, then the pattern $p$ is \textit{unstable}. The single-letter pattern $1$ is trivially stable. Another simple example is the pattern $\underline{11}$, since renaming the letters in a permutation does not affect the number of plateaux. Indeed, both of these examples generalize to arbitrary length.

\begin{fact}\label{FACT:length1}
	For every positive integer $\ell$, the patterns $11\cdots1$ and $\underline{11\cdots1}$ of length $\ell$ are stable.
\end{fact}

The bijective proofs of stability in this paper rely on the following well-known algebraic fact.

\begin{fact}[{\cite[p.~4]{bookSagan}}]\label{FACT:transpositon}
Any permutation in the symmetric group $\mathfrak{S}_n$ can be expressed as a product of the adjacent transpositions $(1,2), (2,3), \ldots, (n-1,n)$, where $(i,i+1)$ denotes the permutation that swaps $i$ and $i+1$ and fixes all other elements.
\end{fact}

Given a multiset $M=M(k_1,k_2,\ldots,k_n)$ and an integer $i\in\{1,2,\ldots,n-1\}$,
we assume, without loss of generality, that $k_1,k_2,\ldots,k_n$ are positive integers unless otherwise stated.
Let $M_i$ denote the multiset $(i,i+1)\cdot M$, i.e., $M_i=M(k_1\ldots,k_{i-1},k_{i+1},k_i,k_{i+2},\ldots,k_n)$. Fact~\ref{FACT:transpositon} implies that, to verify the stability of a pattern $p$, we only need to test the equality of sizes of the sets $M^*(p;s)$ and $M^*_i(p;s)$ for any multiset $M=M(k_1,k_2,\ldots,k_n)$, any $i=1,2,\ldots,n-1$, and any $s\in\bN$. Similarly to the approach in~\cite[Section 4]{Savage2006}, our basic idea is to construct an appropriate bijection between $M^*$ and $M^*_i$ that preserves the distribution of $p$. To this end, we introduce some fundamental bijections. Let $r$ denote the reverse operation on $M^*$, defined by $r(\pi_1\pi_2\cdots \pi_m):=\pi_m\pi_{m-1}\cdots \pi_1$. For any $i=1,2,\ldots,n-1$, the mapping $\tau_i$ from $M^*$ to $M^*_i$ sends a permutation $\pi$ to $(i,i+1)\cdot \pi$. 

\begin{exa}\label{EX:classical-basic-bij}
	Let $M$ be the multiset $M(2,4,2,1)$ and $i=1$. Then we have $M_1=M(4,2,2,1)$. For $\pi=321432212\in M^*$, we have $\tau_{1}(\pi)=312431121\in M_1^*$.
\end{exa}

\section{Stability and classical patterns}\label{SEC:classical}

This section is devoted to a complete classification of stable classical patterns.

The stability of the pattern 21 follows immediately from the classical identity
\[
\sum_{\pi\in (M(k_1,k_2,\ldots,k_n))^*}q^{21(\pi)}
=\begin{bmatrix}k_1+\cdots+k_n\\k_1,\ldots,k_n\end{bmatrix}_q
=\frac{[k_1+\cdots+k_n]_q!}{[k_1]_q!\cdots[k_n]_q!},
\]
which is symmetric in the multiplicities and gives the distribution of $21$~\cite[Item 105]{MacMahon1915}. 
In what follows, we provide an alternative proof based on a bijective construction.

\begin{thm}\label{THM:classical-length2-stable}
	The length-$2$ classical patterns, i.e., the patterns $12$ and $21$, are stable.
\end{thm}
\begin{proof}
	It is sufficient to show that the pattern $12$ is stable. For any multiset $M=M(k_1,k_2,\ldots,k_n)$  with $\displaystyle\sum_{t=1}^{n}k_t=m$, and any $i=1,2,\ldots,n-1$, we first define a mapping $\Psi_i:\,M^*\to M^*_i$. Given a permutation $\pi\in M^*$, write $\pi$ as the concatenation of several factors:
	\begin{align}\label{EQ:i-i+1-decom}
	\pi=x_1X_1x_2X_2\cdots X_{\rho-1}x_{\rho},
	\end{align}
	where $x_1$ and $x_{\rho}$ can be empty, and $X_j$ is a longest run of letters $i$ and $i+1$. Let $r(X_1X_2\cdots X_{\rho-1})=Y_1Y_2\cdots Y_{\rho-1}$, where each $Y_j$ has the same length as  $X_j$ for $j=1,2,\ldots,\rho-1$. We define $\Psi_{i}(\pi)$ as the permutation $\tau_{i}(x_1Y_1x_2Y_2\cdots Y_{\rho-1}x_{\rho})$. Once we show that $\Psi_{i}$ is a bijection between $M^*(12;s)$ and $M_i^*(12;s)$ for all $s\in\bN$, it will follow from Fact~\ref{FACT:transpositon} that the pattern $12$ is stable.
	
	Let $m=|M|$ and $s_1$ is a non-negative integer. For any $\pi^{1}=\pi^{1}_1\pi^1_2\cdots \pi^1_m\in M^*(12;s_1)$, assume that $\pi^2=\Psi_{i}(\pi)=\pi^2_1\pi^2_2\cdots \pi^2_m\in M_i^*(12;s_2)$  for $s_2\in\bN$ and let $$S_k:=\{(j,j')\ |\ \pi^k_j\pi^k_{j'}\text{ is a subsequence of }\pi^k\text{ such that }\pi^k_j<\pi^k_{j'}\}\text{ for }k=1,\,2.$$
	Then the size of $S_k$ is exactly $s_k$ for $k=1,\,2$. Furthermore, let $S_{k,\,1}:=\{(j,j')\in S_k\ |\ \pi^k_j=i,\,\pi^k_{j'}=i+1\}$ and $S_{k,\,2}:=\{(j,j')\in S_k\ |\ \{\pi^k_j,\,\pi^k_{j'}\}\neq\{i,i+1\}\}$ for $k=1,\,2$. Then, $S_k$ is the disjoint union of $S_{k,\,1}$ and $S_{k,\,2}$ for $k=1,\,2$. It is immediate to see that $S_{1,\,2}=S_{2,\,2}$ since the letters $i$ and $i+1$ are indistinguishable with respect to the other letters when forming an occurrence of $12$. Thus, it remains to show that $|S_{1,\,1}|=|S_{2,\,1}|$. Let $X^k$ be the longest subsequence of letters $i$ and $i+1$ in $\pi^k$ for $k=1,\,2$. Then, the set $S_{1,\,1}$ (resp., $S_{2,\,1}$) is given by all the occurrences of $12$ in $X^1$ (resp., $X^2$). Note that $X^2$ can be obtained from $X^1$ by reversing and complementing if we view $X^k$'s as words over the alphabet $\{i,i+1\}$, so the number of the occurrences of $12$ in $X^1$ is equal to that in $X^2$. This means that the sets $S_{1,\,1}$ and $S_{2,\,1}$ are equinumerous, completing the proof.
\end{proof}

We provide an example to illustrate the mapping $\Psi_{i}$ in the proof of Theorem~\ref{THM:classical-length2-stable}.

\begin{exa}\label{EX:-classical-12}
	Let $M=M(2,4,2,1)$, $i=1$, and $\pi=321432212$ be as in Example~\ref{EX:classical-basic-bij}. Writing $\pi$ in the form as that in Equation~\eqref{EQ:i-i+1-decom}, we have $x_1=3$, $X_1=21$, $x_2=43$, $X_2=2212$ and $x_3$ is empty. Since $r(X_1X_2)=212212$ can be written as $Y_1Y_2$, where $Y_1=21$ and $Y_2=2212$, we have $\Psi_{1}(\pi)=\tau_{1}(x_1Y_1x_2Y_2x_{3})=\tau_{1}(321432212)=312431121$. Both $\pi$ and $\Psi_1(\pi)$ have nine occurrences of the pattern $12$. In particular, we have $S_{1,\,2}=S_{2,\,2}=\{(1,4),(2,4),(2,5),(3,4),(3,5)\}$.
\end{exa}

Recall that an \textit{inversion} in a word is an occurrence of the pattern 21.

\begin{cor}
	For any $s\in\bN$, the property ``to have $s$ inversions" is stable. 
\end{cor}

For any classical pattern of length greater than two, we construct a universal example to show that this pattern is unstable.

\begin{lem}\label{LEM:classical-length3-unstable}
	For any classical pattern $p$ of length greater than two, the property ``to have one occurrence of $p$" is unstable.
\end{lem}
\begin{proof}
	Suppose that the pattern $p$ can be written as $p=p_1p_2\cdots p_\ell$ with $\ell\geq3$. It is sufficient to construct two multisets $M=M(k_1,k_2,\ldots,k_\ell)$ and $\overline{M}$, where $\overline{M}$ can be obtained from $M$ by swapping $k_{p_2}$ and $k_{p_\ell}$, and show that the sets $M^*(p;1)$ and $\overline{M}^*(p;1)$ are not equinumerous.
	
	Let $M$ be the multiset $M(k_1,k_2,\ldots,k_\ell)$ with $k_{p_1}=2$, $k_{p_\ell}=3$ and $k_i=1$ for all $i\in\{1,2,\ldots,\ell\}\setminus\{p_1,\,p_\ell\}$. We will prove that $|M^*(p;1)|=(\ell^3-5\ell)/2$. 
	It is easy to see that any permutation in $M^*(p;1)$ must contain a subsequence $p$. Therefore, one can construct a permutation in $M^*(p;1)$ from the permutation $\pi=p$ by inserting one letter $p_1$ and then two letters $p_\ell$'s. The letter $p_1$ must be inserted after the letter $p_2$ in $\pi$ to ensure that the resulting permutation $\pi'$ contains exactly one occurrence of the pattern $p$. For the same reason, two letters $p_\ell$'s must be inserted before the letter $p_{\ell-1}$ in $\pi'$ to obtain the final permutation $\pi''\in M^*(p;1)$. If $\pi'$ is of the form $p_1p_2\cdots p_{\ell-1}p_1p_\ell$ or $p_1p_2\cdots p_{\ell-1}p_\ell p_1$, then there are $\binom{\ell}{2}$ possibilities to insert two $p_\ell$'s; otherwise, there are $\ell-3$ possibilities for $\pi'$ and each $\pi'$ results in $\binom{\ell+1}{2}$ distinct $\pi''$'s. Consequently, there are exactly $2\cdot \binom{\ell}{2}+(\ell-3)\cdot\binom{\ell+1}{2}=\ell(\ell^2-5)/2$ elements in $M^*(p;1)$.
	
	Swapping $k_{p_2}$ and $k_{p_\ell}$, we obtain the multiset $\overline{M}$. Similarly to the above, a permutation in $\overline{M}^*(p;1)$ can be constructed from the permutation $\pi=p$ by inserting one letter $p_1$ and then two letters $p_2$'s. In this case, the only special case for $\pi'$ is of the form $p_1p_2p_1p_3\cdots p_{\ell-1}p_\ell$, contributing $\binom{\ell}{2}$ distinct $\pi''$'s; the other $(\ell-2)$ $\pi'$'s each allow $\binom{\ell+1}{2}$ ways to place the extra two $p_2$'s. Therefore, we have $|\overline{M}^*(p;1)|=1\cdot \binom{\ell}{2}+(\ell-2)\cdot\binom{\ell+1}{2}=\ell(\ell^2-3)/2$, which is never equal to $\ell(\ell^2-5)/2$, the size of the set $M^*(p;1)$. This completes the proof.
\end{proof}

Combining Fact~\ref{FACT:length1}, Theorem~\ref{THM:classical-length2-stable}, and Lemma~\ref{LEM:classical-length3-unstable}, we can fully characterize the stability of all classical patterns.

\begin{thm}\label{THM:classic-pattern-charact}
	A classical pattern is stable if and only if its length is one or two.
\end{thm}

\begin{rem}
	For a pattern $p$ of length greater than two, note that the proof of Lemma~\ref{LEM:classical-length3-unstable} does not extend to any $s\in\bN$ to show that the property ``to have $s$ occurrences of $p$" is unstable. In fact, when $s=0$, the property becomes stable for some $p$. A classical result in this direction is that the property ``to avoid the classical pattern $p$" for any length-$3$ $p$ is stable, see~\cite{Atkinson1995,Albert2001,Savage2006} and~\cite{Zeilberger2005} with a detailed review available in~\cite[Section 3.2]{Shar2016}. We note that the bijection given in~\cite[Section 4]{Savage2006} can be extended to show that the property ``to avoid the pattern $12\cdots\ell$" is stable for any positive integer $\ell$. This result was also established by Albert et al.~\cite{Albert2001} using the Robinson-Schensted-Knuth correspondence between permutations of multisets and Young tableau pairs. The enumeration result of Heubach and Mansour~\cite{Heubach2006} implies a similar result for $p=11\cdots1$ of any positive length.
\end{rem}

\section{Stability and consecutive patterns}\label{SEC:consecutive}

In this section, we begin by giving a bijective proof of the stability of all monotone consecutive patterns. We then complement this with the exhibition of a broad family of unstable ones.

The stability of monotone consecutive patterns follows from the classical run theorem, which gives an explicit expression for the generating function of increasing runs in terms of elementary symmetric functions~\cite{Gessel1977,Goulden1983,Jackson1977}. Subsequently, Yang and Zeilberger presented explicit generating functions for the distribution of increasing consecutive patterns and exploited their symmetry to derive recurrence relations that admit efficient computation. In particular, the generating function for the pattern $\underline{21}$ was given by
\begin{align}\label{EQ:GF-12_}
\nonumber
&\,G_{\underline{21}}(x_0,x_1,x_2,\ldots)=\\
&\,\frac{1}{1-\sum_{i\geq1}x_i+(1-x_0)\sum_{1\leq i_1<i_2}x_{i_1}x_{i_2}-(1-x_0)^2\sum_{1\leq i_1<i_2<i_3}x_{i_1}x_{i_2}x_{i_3}+\cdots}
\end{align}
and dates back to MacMahon~\cite[Item 155]{MacMahon1915} (where occurrences of $\underline{21}$ are referred to as ``major contacts''). In what follows, we provide an alternative bijective proof.

\begin{lem}\label{LEM:cons-length-2}
	The patterns $\underline{12}$ and $\underline{21}$ are stable.
\end{lem}
\begin{proof}
Let $M$ be any multiset $M(k_1,k_2,\ldots,k_n)$ and $i=1,2,\ldots,n-1$. We define a bijection $\Phi_i$ from $M^*$ to $M^*_i$ as follows. Given a permutation $\pi\in M^*$, decompose $\pi$ as in~\eqref{EQ:i-i+1-decom}. The permutation $\Phi_{i}(\pi)$ is given by $\tau_{i}(x_1r(X_1)x_2r(X_2)\cdots r(X_{\rho-1})x_\rho)$. Then replacing the mapping $\Psi_{i}$ in the proof of Theorem~\ref{THM:classical-length2-stable} with $\Phi_{i}$, one can similarly prove that $\Phi_{i}$ is a bijection between $M^*(\underline{12};s)$ and $M_i^*(\underline{12};s)$ for all $s\in\bN$. The lemma then follows from Fact~\ref{FACT:transpositon}.
\end{proof}

This leads to the following stability result for the descent numbers.

\begin{cor}
	For any $s\in\bN$, the property ``to have $s$ descents'' is stable.
\end{cor}

\begin{rem}
	Note that $\Psi_{i}$ may fail to be a bijection between $M^*(\underline{12};s)$ and $M_i^*(\underline{12};s)$ for certain choices of $M$, $i$, and $s$. For example, for $M=M(2,1,1)$ and $i=1$, the permutation $\pi=1132\in M^*(\underline{12};1)$ while the permutation $\Psi_{1}(\pi)=1232\in M^*(\underline{12};2)$. In fact, this example also shows that $\Phi_{i}$ does not work to prove the stability of the pattern $12$, because $\pi=1132\in M^*(12;4)$ while $\Phi_{1}(\pi)=2231\in M^*(12;2)$.
\end{rem}

We now generalize Lemma~\ref{LEM:cons-length-2} to longer patterns. Given a multiset $M=M(k_1,k_2,\ldots,k_n)$ and an integer $i\in\{1,2,\ldots,n-1\}$, we define a mapping $\Theta_i$ for use in the subsequent proof. For a permutation $\pi\in M^*$ with decomposition as in~\eqref{EQ:i-i+1-decom}, we define $\Theta_{i}(\pi):=x_1Y_1x_2Y_2\cdots Y_{\rho-1}x_{\rho}$, where each $Y_j$ is obtained from $X_j$ by replacing the letter $i$ (resp., $i+1$) by $i+1$ (resp., $i$) in each position, except the following three cases:
\begin{enumerate}
	\item The length of $X_j$ is either two or at least four. Moreover, either its first two letters are distinct or its last two letters are distinct, and the letter is one of those two distinct letters.
	\item The length of $X_j$ is three. Moreover, either its first or last two letters are distinct and form an ascent in $X_j$, and the letter is one of those two distinct letters.
	\item The length of $X_j$ is three, and the letter is the first or last letter of $X_j=(i+1)(i+1)i$ or $X_j=(i+1)ii$.
\end{enumerate}
For convenience, we list all possible correspondences between $X_j$'s and $Y_j$'s for the length-$3$ $X_j$'s:
\begin{center}
	\begin{tabular}{|c|c|}
		\hline
		$X_j$           & $Y_j$           \\\hline
		$iii$             & $(i+1)(i+1)(i+1)$ \\
		$(i+1)(i+1)(i+1)$ & $iii$             \\
		$i(i+1)i$         & $i(i+1)(i+1)$     \\
		$(i+1)i(i+1)$     & $ii(i+1)$         \\
		$(i+1)(i+1)i$     & $(i+1)ii$         \\
		$ii(i+1)$         & $(i+1)i(i+1)$     \\
		$i(i+1)(i+1)$     & $i(i+1)i$     \\
		$(i+1)ii$         & $(i+1)(i+1)i$    \\\hline
	\end{tabular}
\end{center}
It is easy to see that $\Theta_{i}$ is a bijection between $M^*$ and $M_i^*$. We illustrate the mapping $\Theta_{i}$ more explicitly by an example.

\begin{exa}
	Let $M=M(2,4,2,1)$, $i=1$, and $\pi=321432212$ be as in Example~\ref{EX:classical-basic-bij}. As shown in Example~\ref{EX:-classical-12}, decomposing $\pi$ according to~\eqref{EQ:i-i+1-decom} yields $x_1=3$, $X_1=21$, $x_2=43$, $X_2=2212$ with empty $x_3$. To obtain $\Theta_{1}(\pi)$, we first write out $Y_j$'s: $Y_1=21$ and $Y_2=1112$. Therefore, we have $\Theta_{1}(\pi)=321431112\neq \Psi_1(\pi)$.
\end{exa}

\begin{thm}\label{THM:mono-cons-stable}
	The monotone consecutive patterns, i.e., the patterns $\underline{12\cdots\ell}$ and $\underline{\ell(\ell-1)\cdots1}$ for any positive integer $\ell$, are stable.
\end{thm}
\begin{proof}
	It is sufficient to prove the stability of the pattern $p=\underline{12\cdots\ell}$ for $\ell>2$. Let $M$ be a multiset $M(k_1,k_2,\ldots,k_n)$ and $i\in\{1,2,\ldots,n-1\}$. Similarly to the proof of Theorem~\ref{THM:classical-length2-stable}, we only need to verify that $\Theta_{i}$ is a bijection between $M^*(p;s)$ and $M_i^*(p;s)$ for any $s\in\bN$.
	
	For any $\pi^{1}=\pi^{1}_1\pi^1_2\cdots \pi^1_m\in M^*$, assume that $\pi^2=\Theta_{i}(\pi)=\pi^2_1\pi^2_2\cdots \pi^2_m\in M_i^*$ and let $$S_k:=\{j\ |\ \pi^k_j\pi^k_{j+1}\cdots \pi^k_{j+\ell-1}\text{ is a strictly increasing factor of }\pi^k\}\text{ for }k=1,\,2.$$
	Furthermore, let $S_{k,\,1}:=\{j\in S_k\ |\ \pi^k_j\pi^k_{j+1}\cdots \pi^k_{j+\ell-1}\text{ contains a factor }i(i+1)\}$ and $S_{k,\,2}:=\{j\in S_k\ |\ \text{at most one letter in }\pi^k_j\pi^k_{j+1}\cdots \pi^k_{j+\ell-1}\text{ is }i \text{ or }i+1\}$ for $k=1,\,2$. Then $S_k$ is the disjoint union of $S_{k,\,1}$ and $S_{k,\,2}$ for $k=1,\,2$. We immediately have $S_{1,\,2}=S_{2,\,2}$. So it remains to show that $S_{1,\,1}=S_{2,\,1}$. 
	
	For any $j\in S_{1,\,1}$, denote the subsequence $\pi^1_j\pi^1_{j+1}\cdots \pi^1_{j+\ell-1}$ by $\pi'$. Writing $\pi^1$ in the form of Equation~\eqref{EQ:i-i+1-decom}, we note that there exists exactly one $X_{j'}$ such that $X_{j'}$ and $\pi'$ overlap. If $\pi'$ includes $X_{j'}$, then $X_{j'}$ must have length two. Otherwise, the overlapping part must be the first or the last two letters of $X_{j'}$ and these two letters will form an ascent. In either case, $\pi'$ is fixed under the mapping $\Theta_{i}$. Therefore, we have $j\in S_{2,\,1}$ and thus $S_{1,\,1}\subseteq S_{2,\,1}$. By similar arguments, we obtain $S_{2,\,1}\subseteq S_{1,\,1}$. This completes the proof.
\end{proof}

\begin{rem}
	The mapping $\Theta_{i}$ is not an appropriate bijection to prove Lemma~\ref{LEM:cons-length-2}. For instance, given $M=M(3,1,1)$ and $i=1$, the permutation $\pi=11213\in M^*(\underline{12};2)$ while the permutation $\Theta_{1}(\pi)=22213$ is in $M^*(\underline{12};1)$. Conversely, $\Phi_{i}$ also fails to prove the stability of the pattern $\underline{123}$, since $\pi'=12113\in M^*(\underline{123};0)$ while $\Phi_1(\pi')=22123\in M^*(\underline{123};1)$. 
	The bijection $\Gamma: M^* \to M^*_i$ of Han~\cite{Han1992}, originally introduced to prove that a statistic $\mathrm{Z}$ is Mahonian, may therefore be a better choice. For $\pi \in M^*$, the map $\Gamma_i(\pi)$ proceeds in three steps. First, replace each occurrence of $(i+1)i$ by a placeholder $\sim$. Second, every remaining factor using only $i$ and $i+1$ has the form
	$$\underbrace{ii\cdots i}_{j_1}\underbrace{(i+1)(i+1)\cdots (i+1)}_{j_2},$$
	and we replace it by
	$$\underbrace{ii\cdots i}_{j_2}\underbrace{(i+1)(i+1)\cdots (i+1)}_{j_1}.$$
	Third, restore each placeholder $\sim$ to $(i+1)i$. It is easy to see that this bijection preserves the joint distribution of all decreasing consecutive patterns, yielding another bijective proof of Theorem~\ref{THM:mono-cons-stable}.
\end{rem}

Next we construct a universal example demonstrating the instability of a broad class of non-monotone consecutive patterns. To this end, we first introduce some notions. Given a consecutive pattern $p=\underline{p_1p_2\cdots p_\ell}$, an index $i$, where $2\leq i\leq\ell$, is called \textit{extendable} if there exists a permutation $\pi$ of length $\ell+i-1$ such that the first and the last $\ell$ letters of $\pi$ each form an occurrence of the pattern $p$. Such a permutation is called an \textit{extended permutation} of the pattern $p$ at the index $i$. Among all such extended permutations $\pi=\pi_1\pi_2\cdots \pi_{\ell+i-1}$, the minimal one with respect to the lexicographic order is termed the \textit{standard extended permutation} of the pattern $p$ at the index $i$. The corresponding multiset of the standard extended permutation is the \textit{extended multiset} of the pattern $p$ at the index $i$. Note that the index $\ell$ is obviously extendable.

\begin{exa}\label{EX:extended-perm}
Let $p=\underline{24135}$. Then the index $4$ is extendable, and the standard extended permutation $\pi=24135146$ can be obtained as shown in Figure~\ref{fig:grid}. Similarly, the index $5$ is extendable, yielding the standard extended permutation $241357168$.
\end{exa}

\begin{figure}[htbp]
	\centering
	\begin{subfigure}[b]{0.45\textwidth}
		\centering
		\begin{tikzpicture}[scale=0.5]
		\draw[->, thick] (0,0) -- (9,0) node[right] {index};
		\draw[->, thick] (0,0) -- (0,7) node[above] {};
		
		\foreach \x in {0,...,8}
		\filldraw (\x,0) circle (1pt);
		
		
		\foreach \x in {1,...,8}
		\draw (\x,0) node[below] {$\pi_{\x}$};
		
		\filldraw (4,3) circle (2pt); 
		
		\filldraw (5,5) circle (2pt); 
		
		\end{tikzpicture}
		\caption{Step 1: draw the overlap points $\pi_4$ and $\pi_5$ in increasing order.}
		\label{fig:a}
	\end{subfigure}
	\hfill
	\begin{subfigure}[b]{0.45\textwidth}
		\centering
		\begin{tikzpicture}[scale=0.5]
		\draw[->, thick] (0,0) -- (9,0) node[right] {index};
		\draw[->, thick] (0,0) -- (0,7) node[above] {};
		
		\foreach \x in {0,...,8}
		\filldraw (\x,0) circle (1pt);
		
		
		\foreach \x in {1,...,8}
		\draw (\x,0) node[below] {$\pi_{\x}$};
		
		\filldraw (4,3) circle (2pt); 
		
		\filldraw (5,5) circle (2pt); 
		
		\filldraw (1,2) circle (2pt);
		
		\filldraw (2,4) circle (2pt);
		
		\filldraw (3,1) circle (2pt);
		
		\end{tikzpicture}
		\caption{Step 2: draw the first three points such that $\pi_1\pi_2\cdots\pi_5$ is an occurrence of $p$.}
		\label{fig:b}
	\end{subfigure}
	
	\vspace{0.5cm}
	
	\begin{subfigure}[b]{0.45\textwidth}
		\centering
		\begin{tikzpicture}[scale=0.5]
		\draw[->, thick] (0,0) -- (9,0) node[right] {index};
		\draw[->, thick] (0,0) -- (0,7) node[above] {};
		
		\foreach \x in {0,...,8}
		\filldraw (\x,0) circle (1pt);
		
		
		\foreach \x in {1,...,8}
		\draw (\x,0) node[below] {$\pi_{\x}$};
		
		\filldraw (4,3) circle (2pt); 
		
		\filldraw (5,5) circle (2pt); 
		
		\filldraw (1,2) circle (2pt);
		
		\filldraw (2,4) circle (2pt);
		
		\filldraw (3,1) circle (2pt);
		
		\filldraw (6,1) circle (2pt);
		
		\filldraw (7,4) circle (2pt);
		
		\filldraw (8,6) circle (2pt);
		
		\end{tikzpicture}
		\caption{Step 3: draw the last three points such that $\pi_4\pi_5\cdots\pi_8$ is an occurrence of $p$.}
		\label{fig:c}
	\end{subfigure}
	\hfill
	\begin{subfigure}[b]{0.45\textwidth}
		\centering
		\begin{tikzpicture}[scale=0.5]
		\draw[->, thick] (0,0) -- (9,0) node[right] {index};
		\draw[->, thick] (0,0) -- (0,7) node[above] {letter};
		
		\foreach \x in {0,...,8}
		\filldraw (\x,0) circle (1pt);
		
		\foreach \y in {0,...,6}
		\filldraw (0,\y) circle (1pt);
		
		\foreach \x in {1,...,8}
		\draw (\x,0) node[below] {$\pi_{\x}$};
		
		\foreach \y in {1,...,6}
		\draw (0,\y) node[left] {$\y$};
		
		\filldraw (4,3) circle (2pt); 
		
		\filldraw (5,5) circle (2pt); 
		
		\filldraw (1,2) circle (2pt);
		
		\filldraw (2,4) circle (2pt);
		
		\filldraw (3,1) circle (2pt);
		
		\filldraw (6,1) circle (2pt);
		
		\filldraw (7,4) circle (2pt);
		
		\filldraw (8,6) circle (2pt);
		
		\end{tikzpicture}
		\caption{Step 4: mark the dots upwards to obtain the standard extended permutation $\pi=24135146$.}
		\label{fig:d}
	\end{subfigure}
	\caption{Steps for obtaining the standard extended permutation $\pi$ of $p=\underline{24135}$ at the index $4$.}
	\label{fig:grid}
\end{figure}
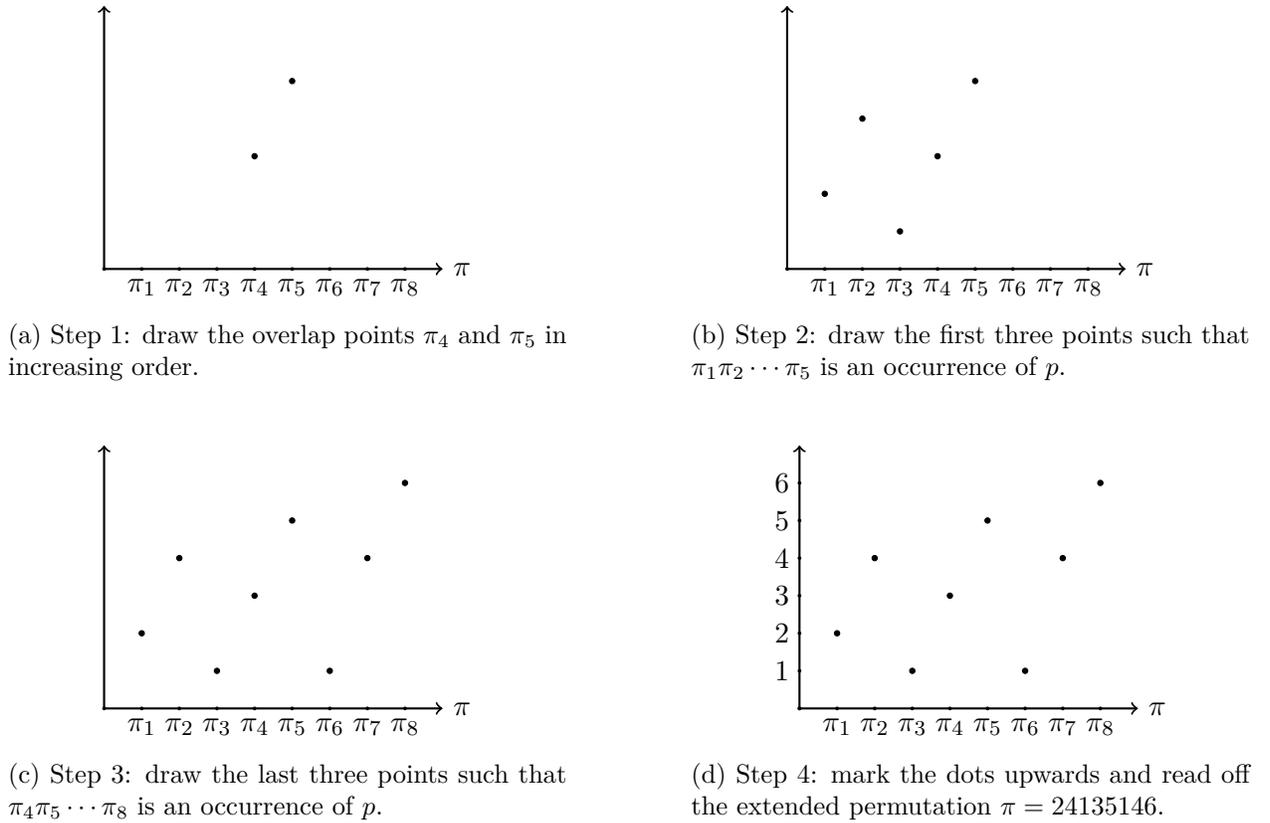

The method to obtain the standard extended permutation in Example~\ref{EX:extended-perm} reveals important counting properties of the extended permutations.

\begin{lem}\label{LEM:ext-like-perm}
	Let $p=\underline{p_1p_2\cdots p_\ell}$ be a consecutive pattern with an extendable index $i$. Let $\sigma\in \fS_{\ell-i+1}$ be a permutation such that $p_{\sigma(1)}<p_{\sigma(2)}<\cdots<p_{\sigma(\ell-i+1)}$. With the convention $p_{\sigma(0)}=p_{\sigma(0)+i-1}=0$ and $p_{\sigma(\ell-i+2)}=p_{\sigma(\ell-i+2)+i-1}=\ell+1$, set 
	$$\delta_j^1:=p_{\sigma(j)}-p_{\sigma(j-1)}-1\quad\text{ and}\quad\delta_j^2:=p_{\sigma(j)+i-1}-p_{\sigma(j-1)+i-1}-1$$
	for $j=1,2,\ldots,\ell-i+2$.
	If $\pi=\pi_1\pi_2\cdots\pi_{\ell+i-1}$ is an extended permutation of $p$ at the index $i$, then \begin{align}\label{EQ:pi-order}
	\pi_{\sigma(1)+i-1}<\pi_{\sigma(2)+i-1}<\cdots<\pi_{\sigma(\ell-i+1)+i-1}.
	\end{align}
	Furthermore, with the convention $\pi_{\sigma(0)+i-1}=0$ and $\pi_{\sigma(\ell-i+2)+i-1}=\infty$, for each $j=1,2,\ldots,\ell-i+2$, there are exactly $\delta_j^1+\delta_j^2$ occurrences of letters in $\pi$ lying strictly between $\pi_{\sigma(j-1)+i-1}$ and $\pi_{\sigma(j)+i-1}$, and at least $\max\{\delta_j^1,\delta_j^2\}$ of them are distinct letters.
\end{lem}
\begin{proof}
	Extendability of the index $i$ implies that the suffix
	$p_ip_{i+1}\cdots p_\ell$ and the prefix $p_1p_2\cdots p_{\ell-i+1}$ have the same reduction.
	This means that $p_{\sigma(1)+i-1}<p_{\sigma(2)+i-1}<\cdots<p_{\sigma(\ell-i+1)+i-1}$ and in particular $\delta_j^2\geq0$ for all $j$. Therefore, $\pi_i,\pi_{i+1},\ldots,\pi_\ell$ are ordered as~\eqref{EQ:pi-order}. With $\pi_i,\pi_{i+1},\ldots,\pi_\ell$ fixed, we consider the first and last $i-1$ letters of $\pi$. Since the first $\ell$ letters of $\pi$ constitute an occurrence of $p$ which has distinct letters, among the first $i-1$ letters, there are exactly $\delta_1^2$ distinct letters smaller than $\pi_{\sigma(1)+i-1}$, $\delta_{\ell-i+2}^2$ distinct letters larger than $\pi_{\sigma(\ell-i+1)+i-1}$, and $\delta_j^2$ distinct letters larger than $\pi_{\sigma(j-1)+i-1}$ and smaller than $\pi_{\sigma(j)+i-1}$ for $j=2,3,\ldots,\ell-i+1$. Considering the last $i-1$ letters of $\pi$, the corresponding enumeration can be obtained by replacing $\delta_j^2$ with $\delta_j^1$ for $j=1,2,\ldots,\ell-i+2$. Combining the counts from the first and last blocks yields the desired enumeration.
\end{proof}

Lemma~\ref{LEM:ext-like-perm} immediately yields a characterization of the extended multiset, since the standard extended permutation is the minimal among all extended permutations.

\begin{prop}\label{PROP:extended-mset-0}
	Let $p=\underline{p_1p_2\cdots p_\ell}$ be a consecutive pattern with an extendable index $i$. Using the notations from Lemma~\ref{LEM:ext-like-perm}, define
	$$\Delta_j^1:=\min\{\delta_j^1,\delta_j^2\}\quad\text{and}\quad\Delta_j^2:=|\delta_j^1-\delta_j^2|$$
	for $j=1,2,\ldots,\ell-i+2$. Then the extended multiset of the pattern $p$ at the index $i$ is \begin{align}\label{EQ:ext-mul}
	M(\underbrace{2,\ldots,2}_{\Delta_1^1},\underbrace{1,\ldots,1}_{\Delta_1^2+1},\underbrace{2,\ldots,2}_{\Delta_2^1},\underbrace{1,\ldots,1}_{\Delta_2^2+1},\ldots,\underbrace{2,\ldots,2}_{\Delta_{\ell-i+1}^1},\underbrace{1,\ldots,1}_{\Delta_{\ell-i+1}^2+1},\underbrace{2,\ldots,2}_{\Delta_{\ell-i+2}^1},\underbrace{1,\ldots,1}_{\Delta_{\ell-i+2}^2}).
	\end{align}
\end{prop}
\begin{proof}
	On the one hand, Lemma~\ref{LEM:ext-like-perm} shows that the multiset $M$ given by~\eqref{EQ:ext-mul} is the minimal possible one. On the other hand, one can construct the standard extended permutation $\pi$ in $M^*$ as follows. First, write the letters in the overlap part by setting $\pi_{\sigma(j)+i-1}=j+\sum_{\mu=1}^{j}(\Delta_\mu^1+\Delta_\mu^2)$ for $j=1,2,\ldots,\ell-i+1$. Next, for the first $i-1$ indices of $\pi$, first consider those whose letters should be smaller than $\pi_{\sigma(1)+i-1}$, and fill them with the letters $1,2,\ldots,\delta_1^2$ in the order induced by the occurrence of $p$; then similarly fill $\delta_2^2$ indices whose letters should be smaller than $\pi_{\sigma(2)+i-1}$ and larger than $\pi_{\sigma(1)+i-1}$; and so on. The filling of the last $i-1$ letters of $\pi$ is analogous.
\end{proof}

In particular, when the pattern $p$ satisfies $p_1<p_\ell$ and we take the extendable index $\ell$, the corresponding extended multiset has a much simpler structure.

\begin{cor}\label{COR:extended-mset}
	Let $p=\underline{p_1p_2\cdots p_\ell}$ be a consecutive pattern. If $p_1<p_\ell$ (resp., $p_1>p_\ell$), then the standard extended permutation of the pattern $p$ at the index $\ell$ begins (resp., ends) with $p$ and the corresponding extended multiset is $M(k_1,k_2,\ldots,k_{\ell+|p_\ell-p_1|})$, where
	\begin{align*}
	k_i=\left\{
	\begin{array}{ll}
	2,& \text{if }1\leq i \leq \min\{p_1,p_\ell\}-1 \text{ or } \max\{p_1,p_\ell\}+1\leq i\leq\ell;\\
	1, &\text{otherwise.}
	\end{array}
	\right.
	\end{align*}
\end{cor}

\begin{exa}\label{EX:extended-mset}
	We illustrate Lemma~\ref{LEM:ext-like-perm}, Proposition~\ref{PROP:extended-mset-0} and Corollary~\ref{COR:extended-mset} with some examples.
	\begin{enumerate}
		\item Let $p=\underline{24135}$. As Example~\ref{EX:extended-perm} shows, the index $4$ is extendable with the standard extended permutation $\pi=24135146$. In fact, Figure~\ref{fig:grid} shows that the points $\pi_4$ and $\pi_5$ divide the diagram into nine regions; as stated in Lemma~\ref{LEM:ext-like-perm}, the remaining points can only lie in six of them, with $\delta_j^i$ counting the number of points in each region, see Figure~\ref{fig:2}. This leads to
		\begin{gather*}
		\Delta_1^1=\min\{1,2\}=1,\,\Delta_2^1=\min\{1,1\}=1,\,\Delta_3^1=\min\{1,0\}=0,\\
		\Delta_1^2=1,\,\Delta_2^2=0,\,\Delta_3^2=1,
		\end{gather*}
		and the corresponding extended multiset $M(2,1,1,2,1,1)$, as shown in Proposition~\ref{PROP:extended-mset-0}.
		\item We have found that the standard extended permutation of the pattern $\underline{24135}$ at the index $5$ is $241357168$, which is consistent with Corollary~\ref{COR:extended-mset}. 
	\end{enumerate}
\end{exa}

\begin{figure}[htbp]
	\centering
	\begin{tikzpicture}[scale=0.8]
	
	\draw[dashed, gray!50, thin] (0,3) -- (9,3); 
	\draw[dashed, gray!50, thin] (0,5) -- (9,5);
	\draw[dashed, gray!50, thin] (4,0) -- (4,7);
	\draw[dashed, gray!50, thin] (5,0) -- (5,7);
	
	\draw[->, thick] (0,0) -- (9,0) node[right] {index};
	\draw[->, thick] (0,0) -- (0,7) node[above] {letter};
	
	\foreach \x in {0,...,8}
	\filldraw (\x,0) circle (1pt);
	
	\foreach \y in {0,...,6}
	\filldraw (0,\y) circle (1pt);
	
	\foreach \x in {1,...,8}
	\draw (\x,0) node[below] {$\pi_{\x}$};
	
	\foreach \y in {1,...,6}
	\draw (0,\y) node[left] {$\y$};
	
	\draw (3,6) node {$\delta_3^2=0$};
	\draw (3,4) node {$\delta_2^2=1$};
	\draw (3,1.5) node {$\delta_1^2=2$};
	\draw (6,6) node {$\delta_3^1=1$};
	\draw (6,4) node {$\delta_2^1=1$};
	\draw (6,1.5) node {$\delta_1^1=1$};
	
	\filldraw (4,3) circle (2pt); 
	
	\filldraw (5,5) circle (2pt); 
	
	\filldraw (1,2) circle (2pt);
	
	\filldraw (2,4) circle (2pt);
	
	\filldraw (3,1) circle (2pt);
	
	\filldraw (6,1) circle (2pt);
	
	\filldraw (7,4) circle (2pt);
	
	\filldraw (8,6) circle (2pt);
	\end{tikzpicture}
	\caption{The number $\delta_j^i$ denotes the number of points in each region.}
	\label{fig:2}
\end{figure}
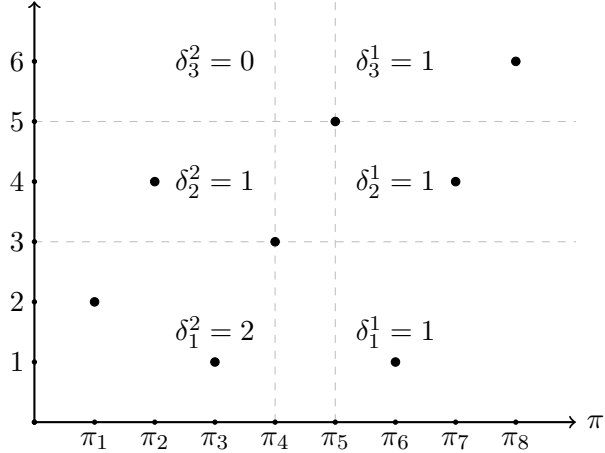

The extended multiset presents a potential example to show that the pattern $p$ is unstable.

\begin{prop}\label{PROP:unstable-by-extended}
	Let $p=\underline{p_1p_2\cdots p_\ell}$ be a consecutive pattern with the minimal extendable index~$i$. If the standard extended permutation of the pattern $p$ at the index $i$ contains repeated letters, then the pattern $p$ is unstable.
\end{prop}

\begin{rem}
	The standard extended permutation of a non-monotone pattern $p$, taken at its minimal extendable index, need not contain repeated letters. For example, the minimal extendable index of $p = \underline{21435}$ is $3$ with the standard extended permutation $2143657$.
\end{rem}

We divide the proof of Proposition~\ref{PROP:unstable-by-extended} into the following two lemmas.

\begin{lem}\label{LEM:unstable-by-extended-1}
	Let $p=\underline{p_1p_2\cdots p_\ell}$ be a consecutive pattern with the minimal extendable index $i$. Following the notations used in Proposition~\ref{PROP:extended-mset-0}, if exactly one of the numbers $\Delta_{1}^1,\Delta_{2}^1,\ldots,\Delta_{\ell-i+2}^1$ is positive, then the property ``to have two occurrences of $p$" is unstable.
\end{lem}
\begin{proof}
	Suppose that the only positive number is $\Delta_{j_0}^1$ for some $1\leq j_0\leq \ell-i+2$. Denote the extended multiset of the pattern $p$ at the index $i$ by $M$. Let $$\ell':=\ell-i+1+\Delta_{j_0}^1+\sum_{j=1}^{\ell-i+2}\Delta_j^2\quad\text{and}\quad\ell_1:=j_0-1+\sum_{j=1}^{j_0-1}\Delta_j^2.$$ By Proposition~\ref{PROP:extended-mset-0}, we have $M=M(k_1,k_2,\ldots,k_{\ell'})$ with $k_{\ell_1+1}=2$. The minimality of the extendable index $i$ implies that the standard extended permutation has exactly two occurrences of the pattern $p$ and thus $|M^*(p;2)|>0$. The proof will be completed by constructing a multiset $\overline{M}=M(k'_1,k'_2,\ldots,k'_{\ell'})$ from $M$ such that $|\overline{M}^*(p;2)|=0$. Depending on whether $j_0=\ell-i+2$, there are two methods for constructing $\overline{M}$.
	
	We first assume that $j_0<\ell-i+2$. Then $k_{\ell'-\Delta_{\ell-i+2}^2}=k_{\ell'-\Delta_{\ell-i+2}^2+1}=\cdots=k_{\ell'}=1$. Combining $k_{\ell_1+1}=2$, we have $\ell_1+1<\ell'-\Delta_{\ell-i+2}^2$. Set $\overline{M}$ to be the multiset obtained from $M$ by swapping $k_{\ell_1+1}$ and $k_{\ell'}$. Suppose that there exists a permutation $\pi=\pi_1\pi_2\cdots \pi_{\ell+i-1}\in \overline{M}^*(p;2)$. By the minimality of the extendable index, the first and the last $\ell$ letters of $\pi$ must constitute two occurrences of $p$. Lemma~\ref{LEM:ext-like-perm} shows that there are at least $\Delta_{\ell-i+2}^2$ distinct
	letters larger than $\pi_{\sigma(\ell-i+1)+i-1}$ and at least $\ell+i-\Delta_{\ell-i+2}^2-2$ letters (counting multiplicities) smaller than $\pi_{\sigma(\ell-i+1)+i-1}$ in $\overline{M}$. The former enumeration indicates that $\pi_{\sigma(\ell-i+1)+i-1}\leq\ell'-\Delta_{\ell-i+2}^2$; while the latter one implies that $\pi_{\sigma(\ell-i+1)+i-1}>\ell'-\Delta_{\ell-i+2}^2,$
	since $$\sum_{j=1}^{\ell'-\Delta_{\ell-i+2}^2}k'_j=\ell+i-1-\sum_{j=\ell'-\Delta_{\ell-i+2}^2+1}^{\ell'}k'_{j}=\ell+i-\Delta_{\ell-i+2}^2-2.$$
	This contradiction means no such letter $\pi_{\sigma(\ell-i+1)+i-1}$ exists in $\overline{M}$, proving $|\overline{M}^*(p;2)|=0$.
	
	For the case of $j_0=\ell-i+2$, let $\overline{M}$ be the multiset derived from $M$ by swapping $k_{1}$ and $k_{\ell_1+1}$. Supposing that there exists a permutation $\pi=\pi_1\pi_2\cdots \pi_{\ell+i-1}\in \overline{M}^*(p;2)$, the proof concludes analogously by showing the nonexistence of the letter $\pi_{\sigma(1)+i-1}$ in $\overline{M}$.
\end{proof}


\begin{lem}\label{LEM:unstable-by-extended-2}
	Let $p=\underline{p_1p_2\cdots p_\ell}$ be a consecutive pattern with the minimal extendable index $i$. Following the notations used in Proposition~\ref{PROP:extended-mset-0}, if there exist $j_1\neq j_2\in\{1,2,\ldots,\ell-i+2\}$ such that $\Delta_{j_1}^1,\,\Delta_{j_2}^1>0$, then the property ``to have two occurrences of $p$" is unstable.
\end{lem}
\begin{proof}
	The basic idea of the proof follows a similar approach to that in Lemma~\ref{LEM:unstable-by-extended-1}.
	Without loss of generality, we 
	can choose minimal indices $j_1<j_2$ satisfying $\Delta_{j_1}^1,\,\Delta_{j_2}^1>0$. Denote the extended multiset of the pattern $p$ at the index $i$ by $M$. Let $$\ell':=\ell-i+1+\sum_{j=1}^{\ell-i+2}(\Delta_j^1+\Delta_j^2),\, \ell_1:=j_1-1+\sum_{j=1}^{j_1-1}\Delta_j^2,\text{ and }\ell_2:=j_2-1+\Delta_{j_1}^1+\sum_{j=1}^{j_2-1}\Delta_j^2.$$ By Proposition~\ref{PROP:extended-mset-0}, we have $M=M(k_1,k_2,\ldots,k_{\ell'})$, where 
	\begin{align*}
	k_1=k_2=\cdots =k_{\ell_1}&=1=k_{\ell_1+\Delta_{j_1}^1+1}=k_{\ell_1+\Delta_{j_1}^1+2}=\cdots=k_{\ell_2},\\ k_{\ell_1+1}=k_{\ell_1+2}=\cdots=k_{\ell_1+\Delta_{j_1}^1}&=2=k_{\ell_2+1}.
	\end{align*}
	Since $p_{\sigma(\ell-i+1)}\neq p_{\sigma(\ell-i+1)+i-1}$, we have $\Delta_{\ell-i+2}^2>0$ and thus $k_{\ell'}=1$. Let $\overline{M}:=M(k'_1,k'_2,\ldots,k'_{\ell'})$ be the multiset obtained from $M$ by swapping $k_{\ell_1+1}$ and $k_{\ell'}$. Since $i$ is the minimal extendable index, the corresponding standard extended permutation is in $M^*(p;2)$. The proof will be completed by showing that $\overline{M}^*(p;2)=\varnothing$.
	 
	Suppose conversely that there exists $\pi=\pi_1\pi_2\cdots \pi_{\ell+i-1}\in \overline{M}^*(p;2)$. By the minimality of the extendable index, the first and the last $\ell$ letters of $\pi$ must form two occurrences of $p$. By Lemma~\ref{LEM:ext-like-perm}, there are at least $$\ell_3:=j_2-2+\sum_{j=1}^{j_2-1}(\delta_j^1+\delta_j^2)=j_2-2+2\Delta_{j_1}^1+\sum_{j=1}^{j_2-1}\Delta_j^2$$ 
	letters (counting multiplicities) smaller than $\pi_{\sigma(j_2-1)+i-1}$ and at least $\ell+i-\ell_3-2$ letters (counting multiplicities) larger than $\pi_{\sigma(j_2-1)+i-1}$ in $\overline{M}$. On the one hand, we have $$\sum_{j=1}^{\ell_2}k'_j=\ell_2+\Delta_{j_1}^1-1=\left(j_2-1+\Delta_{j_1}^1+\sum_{j=1}^{j_2-1}\Delta_j^2\right)+\Delta_{j_1}^1-1=\ell_3,$$
	so $\pi_{\sigma(j_2-1)+i-1}>\ell_2$. On the other hand, the relation $$\sum_{j=\ell_2+2}^{\ell'}k'_{j}=\ell+i-1-\sum_{j=1}^{\ell_2}k'_j-k'_{\ell_2+1}=\ell+i-\ell_3-3<\ell+i-\ell_3-2$$
	implies that $\pi_{\sigma(j_2-1)+i-1}<\ell_2+1$. Therefore, such a letter $\pi_{\sigma(j_2-1)+i-1}$ cannot be found in $\overline{M}$, i.e., $\overline{M}^*(p;2)=\varnothing$.
\end{proof}

\begin{exa}
	We provide some examples to clarify Proposition~\ref{PROP:unstable-by-extended}.
	\begin{enumerate}
		\item Let $p=\underline{3142}$. Then index $3$ is the minimal extendable index of $p$ with standard extended permutation $314253$ and extended multiset $M=M(1,1,2,1,1)$. As shown in the proof of Lemma~\ref{LEM:unstable-by-extended-1}, we set $\overline{M}=M(1,1,1,1,2)$ and prove that $\overline{M}^*(p;2)=\varnothing$. Otherwise, the two occurrences of the pattern $p$ in a permutation $\pi=\pi_1\pi_2\cdots \pi_6\in \overline{M}^*(p;2)$ must be the first four letters and the last four letters. This means that $\pi_1,\pi_2,\pi_4<\pi_3$ and $\pi_4,\pi_6<\pi_3<\pi_5$. Therefore, there is one letter larger than $\pi_3$ and four letters smaller than $\pi_3$, but we cannot find such a letter in $\overline{M}$, which is a contradiction. On the other hand, we have $314253\in M^*(p;2)$, so $|\overline{M}^*(p;2)|=0<|M^*(p;2)|$.
		
		\item Let $p=\underline{24135}$. Then the minimal extendable index of $p$ is $4$. As stated in Example~\ref{EX:extended-mset}~$(i)$, the corresponding extended multiset is $M(2,1,1,2,1,1)$ and $\Delta_1^1=\Delta_2^1=1>0$. Following the proof of Lemma~\ref{LEM:unstable-by-extended-2}, we set $\overline{M}:=M(1,1,1,2,1,2)$, $j_1=1$, and $j_2=2$, and verify that $\overline{M}^*(p;2)=\varnothing$. Otherwise, for a permutation $\pi=\pi_1\pi_2\cdots \pi_8\in \overline{M}^*(p;2)$, the occurrences of $p$ in $\pi$ similarly imply $\pi_1,\pi_3<\pi_4<\pi_2,\pi_5$ and $\pi_6<\pi_4<\pi_5,\pi_7,\pi_8$. Then there are three letters smaller than $\pi_4$ and four letters larger than $\pi_4$, but such a letter cannot be found in $\overline{M}$. So we have $\overline{M}^*(p;2)=\varnothing$. 
	\end{enumerate}
\end{exa}

In particular, Corollary~\ref{COR:extended-mset} and Proposition~\ref{PROP:unstable-by-extended} imply the following corollary.

\begin{cor}\label{COR:unstable-by-extended}
	Let $p=\underline{p_1p_2\cdots p_\ell}$ be a consecutive pattern such that $\{p_1, p_\ell\}\neq\{1,\ell\}$. If the index $\ell$ is the unique extendable index of $p$, then the property ``to have two occurrences of $p$" is unstable and thus the pattern $p$ is unstable.
\end{cor}

\section{Applications of stability: generalized Eulerian numbers}\label{SEC:application}

This section presents examples of applications of the concept of pattern stability. Focusing on the stable pattern $\underline{12}$, we present its distribution over permutations of restricted multisets in terms of the recurrence relations and mixed generating functions.

Let $M=M(k_1,k_2,\ldots,k_n)$ be a multiset and $s$ a non-negative integer. In the special case where $k_1=k_2=\cdots=k_n=1$, it is well-known that the classical \textit{Eulerian number} $E_{n,\,s}$ counts the size of $M^*(\underline{12};s)$~\cite{MacMahon1915,Carlitz1954}. For a general multiset $M$, MacMahon~\cite[Item 155]{MacMahon1915} showed that the number $|M^*(\underline{12};s)|$ equals the coefficient of the term $x_0^sx_1^{k_1}x_2^{k_2}\cdots x_n^{k_n}$ in the expression
\begin{align*}
	\left(x_1+x_0\left(x_2+\cdots+ x_n\right)\right)^{k_1}\left(x_1+x_2+x_0\left(x_3+\cdots +x_n\right)\right)^{k_2}\cdots\left(x_1+x_2+\cdots +x_n\right)^{k_n}
\end{align*}
and obtained the generating function $G_{\underline{12}}$ as shown in Equation~\eqref{EQ:GF-12_}. Both formulas provide explicit ways to compute the size of the set $M^*(\underline{12};s)$.
Theorem~\ref{THM:mono-cons-stable} implies that $|M^*(\underline{12};s)|$ only depends on the multiset $K:=\{k_1,k_2,\ldots,k_n\}$ and the integer $s$. Therefore, we can denote $|M^*(\underline{12};s)|$ by $A_{K,\,s}$. For Eulerian numbers, a fundamental recurrence relation is that
\begin{align}\label{EQ:rec-Eml}
E_{m,\,s}=(s+1)E_{m-1,\,s}+(m-s)E_{m-1,\,s-1}\text{  for any $m,\,s\geq1$},
\end{align}
see, for example,~\cite{Petersen2015}. 
How to count the size of $M^*(\underline{12};s)$ for general multisets is known as the Simon Newcomb problem. In his study of $r$-descent numbers and $r$-major index, Rawlings presented a recurrence for them that includes~\eqref{EQ:rec-Eml} as a special case, see~\cite[Equation (1.8)]{Rawlings1981}. Setting $q=r=1$ and assuming $k_n>1$, the recurrence reduces to
\[k_nA_{K,\,s}=(s+k_n)A_{K',\,s}+\left(\sum_{i=1}^nk_i-k_n-s+1\right)A_{K',\,s-1},\]
where $K'$ is the multiset $\{k_1,k_2,\ldots,k_{n-1},k_{n}-1\}$. The stability further implies the following theorem.

\begin{thm}
	Let $s$ be a positive integer and $K=\{k_1,k_2,\ldots,k_n\}$ be a non-empty multiset with elements in $\bP$. Then for any $j=1,2,\ldots,n$ with $k_j>1$, we have
	\begin{align}\label{EQ:rec-A_Ks-gen-E}
	k_jA_{K,\,s}=(s+k_j)A_{\hat{K}_j,\,s}+\left(\sum_{i=1}^nk_i-k_j-s+1\right)A_{\hat{K}_j,\,s-1},
	\end{align}
	where $\hat{K}_j$ is the multiset $\{k_1,\ldots,k_{j-1},k_j-1,k_{j+1},\ldots,k_{n}\}$.
\end{thm}



To obtain other recurrence relations, we restrict ourselves to the case where $k_i\in\{1,2\}$ for $i=1,2,\ldots,n$. In this case, the size of $M^*(\underline{12};s)$ only depends on $s$, the number $k$ of $2$'s in $K$, and the size $m:=|M|$. This allows us to simplify the notation, writing $A_{m,\,k,\,s}$ for $A_{M(m-2k,k),\,s}$. 

\begin{thm}\label{THM:rec-A_mkl}
	Let $m,\,k,\,s$ be non-negative integers. We have
	\begin{align}\label{EQ:rec-A_mkl}
	A_{m,\,k,\,s}=\frac{1}{2}\left(A_{m,\,k-1,\,s}+A_{m-1,\,k-1,\,s}-A_{m-1,\,k-1,\,s-1}\right)
	\end{align}
	for any $m,\,s\geq1$ and $1\leq k\leq m/2$. Furthermore, we have $A_{m,\,k,\,s}=0$ for any $m,\,s\geq0$ and $k>m/2$; $A_{m,\,k,\,0}=1$ for any $m\geq0$ and $0\leq k\leq m/2$; $A_{m,\,0,\,s}=E_{m,\,s}$ for any $m,\,s\ge 0$.
\end{thm}
\begin{proof}
	The initial values and boundary conditions are straightforward, so we focus on the proof of Equation~\eqref{EQ:rec-A_mkl}. Suppose that $m\geq1$, $1\leq k\leq m/2$, and $s\geq1$. By Theorem~\ref{THM:mono-cons-stable}, we can further assume that $A_{m,\,k,\,s}=|M^*(\underline{12};s)|$ with the multiset $M=\{1,1,2,2,\ldots,k,k,k+1,k+2,\ldots,m-k\}$. Consider the multiset $\overline{M}=\{1,1,2,2,\ldots,k,k^+,k+1,k+2,\ldots,m-k\}$, where $k^+$ denotes a number satisfying $k<k^+<k+1$. Then there are two ways to obtain a permutation in $M^*(\underline{12};s)$ from that in $\overline{M}^*$. One is to choose a permutation in $\overline{M}^*(\underline{12};s)$ where the letters $k$ and $k^+$ are not next to each other and then replace $k^+$ by $k$. Note that in this way, a permutation in $M^*(\underline{12};s)$ can be obtained from exactly two distinct ones in $\overline{M}^*(\underline{12};s)$. The other is to choose a permutation in $\overline{M}^*(\underline{12};s)$ with $k^+k$ as a factor and then replace $k^+$ by $k$. In this way, the permutations in the two sets are in one-to-one correspondence. Let 
	$$M^1:=\{\pi\in \overline{M}^*(\underline{12};s)\ |\ \text{the letters $k$ and $k^+$ are not next to each other in }\pi\},$$
	$$M^2:=\{\pi\in \overline{M}^*(\underline{12};s)\ |\ \pi\text{ has a factor $k^+k$}\},\text{ and }M^3:=\{\pi\in \overline{M}^*(\underline{12};s)\ |\ \pi\text{ has a factor $kk^+$}\}.$$
	Then $\overline{M}^*(\underline{12};s)$ is the disjoint union of the subsets $M^1$, $M^2$, and $M^3$, and 
	\begin{align}\label{eq:A-M}
	A_{m,\,k,\,s}=\frac{1}{2}|M^1|+|M^2|.
	\end{align} 
	We count the size of $M^2$ first. Deleting the letter $k^+$, we obtain a bijection between the sets $M^2$ and $\widetilde{M}^*(\underline{12};s)$, where $\widetilde{M}=\{1,1,2,2,\ldots,k-1,k-1,k,k+1,\ldots,m-k\}$. Therefore, we have 
	\begin{align}\label{eq:M2}
	|M^2|=A_{m-1,\,k-1,\,s}. 
	\end{align}
	Similarly, we have $|M^3|=A_{m-1,\,k-1,\,s-1}$. Note that $|\overline{M}^*(\underline{12};s)|=A_{m,\,k-1,\,s}$, so we have 
	\begin{align}\label{eq:M1}
	|M^1|=|\overline{M}^*(\underline{12};s)|-|M^2|-|M^3|=A_{m,\,k-1,\,s}-A_{m-1,\,k-1,\,s}-A_{m-1,\,k-1,\,s-1}.
	\end{align} 
	Combining Equations~\eqref{eq:A-M},~\eqref{eq:M2}, and~\eqref{eq:M1}, we obtain the desired Equation~\eqref{EQ:rec-A_mkl}.
\end{proof}

Recall that the exponential generating function of the Eulerian numbers $E_{m,s}$ is 
\begin{align}\label{EQ:E(x,y)}
\cE(x,y):=\sum_{m,\,s\ge0}\frac{E_{m,\,s}}{m!}x^my^s=\frac{(1-y)e^{x(1-y)}}{1-ye^{x(1-y)}}.
\end{align}
Using Theorem~\ref{THM:rec-A_mkl}, the mixed generating function
\begin{align}\label{EQ:def-A(x,y,z)}
\cA(x,y,z):=\sum_{m,\,s\ge0}\sum_{k=0}^{\lfloor\frac{m}{2}\rfloor}\frac{A_{m,\,k,\,s}}{(m-2k)!}x^{m-2k}y^kz^s,
\end{align}
for $A_{m,\,k,\,s}$ can be described by a partial differential equation. 

\begin{thm}\label{THM:A-pde}
The mixed generating function $\cA(x,y,z)$ satisfies the partial differential equation
\[
\cA(x,y,z)
= \frac{y}{2}\frac{\partial^2 \cA(x,y,z)}{\partial x^2}
+ \frac{y}{2}(1-z)\frac{\partial \cA(x,y,z)}{\partial x}
+ \frac{(1-z)e^{x(1-z)}}{1 - z e^{x(1-z)}}.
\]
\end{thm}

\begin{proof}
	By Theorem~\ref{THM:rec-A_mkl} and Equation~\eqref{EQ:E(x,y)}, we have 
	\begin{align}
		\cA(x,y,z)=&\sum_{m,\,s\ge0}\sum_{k=0}^{\lfloor\frac{m}{2}\rfloor}\frac{A_{m,\,k,\,s}}{(m-2k)!}x^{m-2k}y^kz^s
		=\sum_{m,\,k,\,s\ge0}\frac{A_{m+2k,\,k,\,s}}{m!}x^{m}y^kz^s\nonumber\\
		=&\sum_{m\ge0,\,k,\,s\ge1}\frac{A_{m+2k,\,k,\,s}}{m!}x^my^kz^s+\sum_{m,\,s\ge0}\frac{A_{m,\,0,\,s}}{m!}x^mz^s+\sum_{m\ge0,\,k\ge1}\frac{A_{m+2k,\,k,\,0}}{m!}x^my^k\nonumber\\
		=&\sum_{m\ge0,\,k,\,s\ge1}\frac{A_{m+2k,\,k,\,s}}{m!}x^my^kz^s+\frac{(1-z)e^{x(1-z)}}{1-ze^{x(1-z)}}+\frac{ye^x}{1-y}.\label{eq:rec-for-Axyz}
	\end{align}
	For the first term in~\eqref{eq:rec-for-Axyz}, Theorem~\ref{THM:rec-A_mkl} implies that
	\begin{align}
	&\sum_{m\geq0,\,k,\,s\ge1}\frac{A_{m+2k,\,k,\,s}}{m!}x^my^kz^s\nonumber\\
	=&\sum_{m\geq0,\,k,\,s\ge1}\frac{1}{2\cdot m!}\left(A_{m+2k,\,k-1,\,s}-A_{m+2k-1,\,k-1,\,s-1}+A_{m+2k-1,\,k-1,\,s}\right)x^my^kz^s\nonumber\\
	=&\sum_{m\geq2,\,k\ge0,\,s\ge1}\frac{A_{m+2k,\,k,\,s}}{2\cdot (m-2)!}x^{m-2}y^{k+1}z^s
	-\sum_{m\geq1,\,k,\,s\ge0}\frac{A_{m+2k,\,k,\,s}}{2\cdot (m-1)!}x^{m-1}y^{k+1}z^{s+1}\nonumber\\
	&+\sum_{m,\,s\ge1,\,k\ge0}\frac{A_{m+2k,\,k,\,s}}{2\cdot (m-1)!}x^{m-1}y^{k+1}z^s\nonumber\\
	=&\,\frac{y}{2}\left(\frac{\partial^2 \cA(x,y,z)}{\partial x^2}-\sum_{m,\,k\ge0}\frac{A_{m+2k+2,\,k,\,0}}{m!}x^my^k\right)
	-\frac{yz}{2}\frac{\partial \cA(x,y,z)}{\partial x}\nonumber\\
	&+\frac{y}{2}\left(\frac{\partial \cA(x,y,z)}{\partial x}-\sum_{m,\,k\ge0}\frac{A_{m+2k+1,\,k,\,0}}{m!}x^my^k\right)\nonumber\\
	=&\,\frac{y}{2}\frac{\partial^2 \cA(x,y,z)}{\partial x^2}
	+\frac{y}{2}(1-z)\frac{\partial \cA(x,y,z)}{\partial x}
	-\frac{ye^x}{1-y}.\label{eq:first-term}
	\end{align}
	Substituting Equation~\eqref{eq:first-term} into Equation~\eqref{eq:rec-for-Axyz}, we obtain the desired result.
\end{proof}

We note that solving the partial differential equation in Theorem~\ref{THM:A-pde} using Maple is possible, but it results in a cumbersome expression for $\cA(x,y,z)$ involving integrals, so we omit it here.

\section{Conclusion}\label{SEC:conclusion}

We conclude by proposing several open conjectures and problems for future research.

We have provided a complete description of the stability of classical patterns. However, the classification for consecutive patterns remains incomplete. In fact, despite extensive computer experiments, we have not found any stable non-monotone consecutive patterns. We therefore propose the following conjecture.

\begin{conj}\label{CONJ:conse-pattern-charact}
	A consecutive pattern is stable if and only if it is monotone.
\end{conj}

Throughout this paper, we focus primarily on classical patterns and consecutive patterns with distinct letters. For patterns with repeated letters, or vincular patterns, computer experiments suggest that the property ``to have exactly one occurrence of pattern $p$" is unstable for the following cases:
\begin{itemize}
	\item the vincular patterns: $1\underline{23}$, $1\underline{32}$, and $2\underline{13}$;
	\item the patterns with plateaux: $112$, $\underline{112}$, $1\underline{12}$, and $\underline{11}2$.
\end{itemize}
Thus, known stable patterns are quite rare, consisting only of patterns of length two, one-letter patterns of arbitrary length, and monotone consecutive patterns. This motivates the following problem.

\begin{prob}
Find new stable patterns, or prove that none exist.
\end{prob}

One reason for the scarcity of stable patterns is that the definition of stability is particularly stringent, as it requires the property to hold over all permutations of all multisets. A natural relaxation is to explore stability under restricted permutation conditions. For example, Yan et al.~\cite{Yan2022} showed that the property ``to be a quasi-Stirling permutation with $i$ descents and $j$ plateaux" is stable for any integers $i,j\geq0$. Alternatively, we can consider stability under a subgroup of the symmetric group. For instance, by considering the subgroup where permutations fix certain letters, the property of unimodality may become stable.

Yet another way to define stability is to consider its refined version: we say that a  pattern $p$ is {\em $s$-stable} if the number of permutations over a multiset with $s$ occurrences of $p$ is independent of a permutation of letters of the multiset. The case of $s=0$ corresponds to avoidance of $p$ and appears in some properties discussed in this paper. We note that our universal examples in this paper prove that the classical patterns of length more than two are not $1$-stable and certain non-monotone consecutive patterns are not $2$-stable, but they say nothing about the cases for $s=3,4,\ldots$.

\begin{prob}
	Classify the $s$-stability of patterns considered in this paper for $s\geq 0$.
\end{prob}

In Section~\ref{SEC:application}, we utilize the stability of pattern $\underline{12}$ to simplify enumeration parameters and derive recurrence relations for the distribution of $\underline{12}$ on permutations of multisets. However, these relations are insufficient to compute arbitrary $A_{K,\,s}$. This leads to the following problem, which may be approachable by extending the proof of Theorem~\ref{THM:rec-A_mkl} to general multisets.
\begin{prob}
	Find other recurrence relations allowing recursive computation of arbitrary $A_{K,\,s}$.
\end{prob}

\section*{Acknowledgments} We are grateful to Zhicong Lin for his valuable suggestions.

\section*{Declaration of AI Assistance}

During the preparation of this manuscript, the authors utilized an AI-assisted tool for language polishing and grammar checking to improve the readability and clarity of the text. All core reasoning and key conclusions were independently completed by the authors. The use of AI did not involve the substantive generation or creative contribution to the research content. The authors take full academic responsibility for the entire manuscript.

\bibliographystyle{plain}

\end{document}